\documentclass{amsart}

\usepackage[english]{babel}
\usepackage[T1]{fontenc}
\usepackage[utf8]{inputenc}
\usepackage{amsmath}
\usepackage{array}
\usepackage{amsthm}
\usepackage{amssymb}
\input xy
\xyoption{all}

\usepackage[cmtip,all]{xy}

\usepackage{stmaryrd}

\usepackage{tikz}
\usetikzlibrary{matrix,arrows,decorations.pathmorphing}
\usepackage{tikz-cd}

\usepackage{enumerate}
\usepackage{mathtools}

\usepackage{xcolor}

\usepackage[backend=bibtex]{biblatex} 
\newtheorem{thmintro}{Theorem}

\newtheorem{printro}[thmintro]{Proposition}
\newtheorem{lmintro}[thmintro]{Lemma}

\newtheorem{thm}{Theorem}[section]

\newtheorem{lm}[thm]{Lemma}

\theoremstyle{definition}
\newtheorem{defiintro}[thmintro]{Definition}
\newtheorem{rkintro}[thmintro]{Remark}
\newtheorem{questionintro}[thmintro]{Questions}
\newtheorem{exintro}[thmintro]{Example}

\newtheorem{defi-prop}[thm]{Proposition-Definition}

\theoremstyle{remark}
\newtheorem{rk}[thm]{Remark}

\newcommand{\Mod}{\mathrm{Mod}}

\newcommand{\A}{{\mathcal{A}}}
\newcommand{\B}{{\mathcal{B}}}
\newcommand{\C}{{\mathcal{C}}}

\newcommand{\D}{{\mathcal{D}}}
\newcommand{\F}{{\mathcal{F}}}

\newcommand{\I}{{\mathcal{I}}}

\newcommand{\G}{{\mathcal{G}}}

\newcommand{\T}{{\mathcal{T}}}
\newcommand{\U}{{\mathcal{U}}}

\newcommand{\Si}{\mathfrak{S}}
\newcommand{\id}{\mathrm{id}}

\newcommand{\colim}{\mathop{\mathrm{colim}}}

\newcommand{\ev}{\mathrm{ev}\,}

\newcommand{\Ext}{\mathrm{Ext}}

\newcommand{\Hom}{\mathrm{Hom}}
\newcommand{\Vect}{\mathbf{Vect}}
\newcommand{\End}{\mathrm{End}}

\newcommand{\kk}{\Bbbk}
\newcommand{\PP}{\mathcal{P}}

\usepackage[foot]{amsaddr}

\title{The universal property of strict polynomial functors}
\author{Antoine Touz\'e}

\address{Univ. Lille, CNRS, UMR 8524 - Laboratoire Paul Painlevé, F-59000 Lille, France}
\thanks{Antoine Touz\'e acknowledges the support of the CDP C2EMPI, together with the French State under the France-2030 programme, the University of Lille, the Initiative of Excellence of the University of Lille, the European Metropolis of Lille for their funding and support of the R-CDP-24-004-C2EMPI project. }

\date{June 2026}

\begin{document}

\begin{abstract}
In characteristic zero, the category of strict polynomial functors is well-known to be the tensor abelian category freely generated by one object. We show that this property fails in positive characteristic, but that it can be repaired by restricting the class of tensor abelian categories considered. The new universal property recovers several known constructions and shows that $\Ext$-algebras of strict polynomial functors act on cohomological computations in many other contexts.
\end{abstract}
\maketitle

\section{Introduction}

Let $\kk$ be a field, and let $\PP$ be the abelian category of strict polynomial functors over $\kk$, introduced by Friedlander and Suslin in \cite{FS}. If $\F$ denotes the abelian category of endofunctors of finite dimensional vector spaces, there is a (faithful, exact) forgetful functor $\U:\PP\to \F$, hence strict polynomial functors can be thought of as endofunctors of finite dimensional vector spaces equipped with an additional structure. Many endofunctors of finite dimensional vector spaces can be interpreted as strict polynomial functors, such as the $d$-th exterior power $\Lambda^d$, the $d$-th symmetric power $S^d$, the $d$-th divided power $\Gamma^d$, or more generally the Schur functors and Weyl functors introduced by Akin Buchsbaum and Weyman \cite{ABW,Wey}. In positive characteristic, homological computations in $\PP$ are a basic building block for homological computations in many contexts, see e.g. \cite{FS,FFSS,FF,TvdK,Touadv,Tousurvey,DTdm,RV}.

The category $\PP$ is naturally endowed with a tensor product operation which makes it a "tensor abelian category". Here we use these polysemic terms in the sense of Buan, Krause and Solberg \cite{BKS}. Namely, a \emph{tensor category} is a $\kk$-linear category, equipped with a symmetric monoidal product which is linear with respect to each of its arguments. 
A \emph{tensor functor} between two tensor categories is a symmetric monoidal linear functor. A \emph{morphism of tensor functors} is a natural transformation of monoidal functors. 
A \emph{tensor abelian category} is tensor category which is abelian and whose symmetric monoidal product is exact with respect to each argument. 
When $\kk$ is a field of characteristic zero $\PP$ free tensor abelian category generated $I=\Lambda^1=S^1=\Gamma^1$. This is a well-know result, which can be found under various forms in the literature see e.g. \cite{BMT}, \cite[Prop 3.3]{DM} or \cite[(6.5.1)]{SS}, and that we state as the following universal property.
\begin{printro}\label{printro:univpropcharzero}
Assume that $\mathrm{char}\,\kk=0$. For all tensor abelian categories $\A$ and for all objects $A$ of $\A$ there is a unique (up to isomorphism) exact tensor functor $\Phi:\PP\to A$ such that $\Phi(I)=A$.   
\end{printro}
The proof of proposition \ref{printro:univpropcharzero} uses in an essential way that $\PP$ is semi-simple, which is specific to characteristic zero. The following example shows that proposition \ref{printro:univpropcharzero} fails in positive characteristic.
\begin{exintro}\label{exintro:obstruction}
Assume that $\mathrm{char}\,\kk=p> 2$, and let $s\Vect$ be the tensor abelian category of super vector spaces. There is no exact tensor functor $\Phi: \PP\to s\Vect$ such that $\Phi(I)=\kk^{0|1}$, the purely odd vector space of dimension one. 
Indeed, assume the contrary. Then $\Phi(1-\tau):\Phi(I^{\otimes 2})\to \Phi(I^{\otimes 2})$ is multiplication by $2$ (where $\tau$ is the symmetry operator in $\PP$), which implies that $\Phi(S^2)=0$. 
This implies in turn that $\Phi$ vanishes on all the Schur functors $S_\lambda$ indexed by partitions $\lambda$ of $p$, since every $S_\lambda$ with $\lambda_1>1$ is a subquotient of $S^2\otimes I^{\otimes p-2}$ by the Pieri formula, and since the only remaining Schur functor $S_{(1^p)}$ is a quotient of $S_{(2,1^{p-2})}$. But by the Pieri formula again, $I^{\otimes p}$ is, up to a filtration, isomorphic to a direct sum of Schur functors indexed by partitions of $p$, hence $\Phi$ 
vanishes on $I^{\otimes p}$, which is impossible.
\end{exintro}

The purpose of the present article is to provide a replacement for proposition \ref{printro:univpropcharzero} which is valid in all characteristics. 
As a preliminary result, we will prove the next lemma, which shows that $\PP$ is generated by $I$, regardless the characteristic of $\kk$.

\begin{lmintro}\label{lmintro:unique}
Let $\A$ be a tensor abelian category and let $\Phi,\Psi:\PP\to \A$ be two exact tensor functors. Every morphism $f:\Phi(I)\to \Psi(I)$ extends to a unique morphism of tensor functors  $\overline{f}:\Phi\to \Psi$. In particular, given an object $A$ of $\A$, there is at most one (up to isomorphism) exact tensor functor $\Phi:\PP\to \A$ such that $\Phi(I)=A$.   
\end{lmintro}

Thus, the obstruction to a universal property in all characteristics is really a problem of $\PP$ being \emph{freely} generated by $I$. In order to bypass this obstruction, we introduce a class of tensor abelian categories which is more closely related to the usual tensor abelian category $\Vect$ of all vector spaces.

\begin{defiintro}\label{defiintro:modeled}
We say that a tensor abelian category $\A$ is \emph{modeled on vector spaces} if there exists a family $(e_x)_{x\in X}$ of tensor functors $e_x:\A\to \Vect$ such that a diagram $0\to A'\to A\to A''\to 0$ in $\A$ is a short exact sequence if and only if 
\[0\to e_xA'\to e_xA\to e_xA''\to 0\]
is a short exact sequence of vector spaces for all $x\in X$. (We shall refer to such a family as a \emph{family of test functors}.)
\end{defiintro}

Though restricted, the class of tensor categories modeled on vector spaces contains many examples of interest, such as:
\begin{enumerate}[(1)]
\item the category $\mathrm{Sh}(X)$ of sheaves of vector spaces on a topological space $X$ (as a family of test functors, one can take the stalks at every point $x\in X$),
\item the category $B-\Mod$ of modules over a cocommutative $\kk$-bialgebra $B$ (with a single test functor given by taking the underlying vector space of a module)
\item The category $B-\mathrm{Comod}$ of comodules over a commutative $\kk$-bialgebra $B$ 
(with a single test functor given by taking the underlying vector space of a module)
\item the category $\C-\Mod$ of functors from a small category $\C$ to vector spaces, with tensor product computed objectwise $(F\otimes G)(x)=F(x)\otimes G(x)$ (a family of test functors is given by the evaluations $e_x:F\mapsto F(x)$ for all objects $x$ in $\C$),
\item the category $\PP$ itself (a family of test functors is given by the evaluations $e_n:F\mapsto F(\kk^n)$ for all $n\ge 1$),
\item any monoidal abelian subcategory of the previous ones, or any graded version of the previous ones provided the gradings are concentrated in even degrees (in order to avoid Koszul signs in the symmetry operator). 
\end{enumerate}

The first main result of this article is the following universal property, which asserts that $\PP$ is freely generated by $I$ as a tensor abelian category modeled on vector spaces. 
\begin{thmintro}\label{thmintro:main}
Let $\A$ be a tensor abelian category modeled on vector spaces, and let $A$ be an object of $\A$. There is a unique (up to isomorphism) exact tensor functor $\Phi:\PP\to A$ such that $\Phi(I)=A$.  
\end{thmintro}
Theorem \ref{thmintro:main} gives a conceptual construction of several tensor exact functors that can be found in the literature. For example, when $\A=\U_2$ is the category of unstable modules at the prime $2$ (this is a full tensor abelian subcategory of the category of modules over the Steenrod algebra) and $A=F(1)$ is the standard projective, we recover the functor $\overline{m}$ constructed in \cite{Hai}. When $\A=\PP^{\mathrm{ev}}$ is the category of graded strict polynomial functors concentrated in even degrees, and $A=V\otimes I$, where $V$ is a finite dimensional graded vector space concentrated in even degrees, we obtain an alternative construction of the parametrization functor introduced in \cite{Touens}, which plays a fundamental role in computing cohomology of classical groups with Frobenius twisted coefficients, see \cite{Tousurvey} for a survey.     

Since the functor $\Phi$ given by theorem \ref{thmintro:main} is exact, it induces a morphism of graded vector spaces $\Ext^*_\PP(F,G)\to \Ext^*_\A(\Phi(F),\Phi(G))$ for all strict polynomial functors $F$ and $G$. This implies that $\Ext$-algebras in $\PP$ act on many $\Ext$ computed in other tensor abelian categories. For instance graded vector spaces $\Ext^*_\A(B,\Phi(G))$ are modules over $\Ext^*_\PP(G,G)$. The interest of this observation lies in the fact that much is known about concrete $\Ext$-computations in $\PP$, see e.g. \cite{FS,FFSS,Tousurvey}.

\subsection{Faithfulness of $\Phi$ and tensor ideals of $\PP$}
In order to study faithfulness of the exact tensor functors contructed by theorem \ref{thmintro:main}, we rely on the categories of strict polynomial functors of bounded domain recently introduced in \cite{CJ}. Thus $\PP^{\le n}$ is the tensor abelian category obtained as the quotient of $\PP$ by the tensor ideal $\I_n$ of the strict polynomial functors satisfying $F(\kk^{n})=0$. We denote by $\pi_n:\PP\to \PP^{\le n}$ the quotient functor. For notational convenience, we also denote by $\I_\infty$ the zero ideal, hence $\pi_\infty:\PP\to \PP^{\le \infty}$ is an equivalence of tensor abelian categories.
\begin{printro}\label{printro:factorization}
Let $\A$ be a category modeled on vector spaces. For all exact tensor functors $\Phi:\PP\to \A$ we have $\mathrm{Ker}\,\Phi =\I_n$ for some $n\in\mathbb{N}\cup\{+\infty\}$. Hence $\Phi$ factors through $\pi_n$ and the quotient $\overline{\Phi}:\PP^{\le n}\to \A$ is a faithful exact tensor functor.
\end{printro}

Notice that $\PP$ has a wealth of tensor ideals different from the ideals $\I_n$. For example every strict polynomial functor $F$ generates an ideal $\I_F$ which equals $\I_n$ if and only if $F\simeq \Lambda^n$. Hence proposition \ref{printro:factorization} shows that exact tensor functors with codomain modeled on vector spaces have a rigid structure. It also shows that for all ideals $\I\ne \I_n$, the quotient category $\PP/\I$ is not modeled on vector spaces (since the kernel of the tensor exact functor $\PP\to \PP/\I$ is different from $\I_n$).
\begin{rkintro}
The factorization $\Phi=\overline{\Phi}\circ \pi_n$ implies that the action of $\Ext$-algebras in $\PP$ on extensions in $\A$ factors through $\PP^{\le n}$. This gives an additional motivation for the study of Ext-algebras in $\PP^{\le n}$ started in \cite{CJ}.   
\end{rkintro}


\subsection{Admissible tensor abelian categories} The class of tensor abelian categories modeled on vector spaces is rather concrete, but far from optimal. We now try to replace it with a larger class of tensor abelian categories for which the universal property remains valid.
\begin{defiintro}
We say that a tensor abelian category $\A$ is \emph{admissible} if for all objects $A$ of $\A$ there exists an exact tensor functor $\Phi:\PP\to \A$ with $\Phi(I)=A$.
\end{defiintro}

Theorem \ref{thmintro:main} asserts that tensor abelian categories modeled on vector spaces are admissible. Our second main result is the following theorem, which provides numerous additional admissible categories. For example, theorem \ref{thmintro:secondmain}\eqref{item2} shows that all the quotients $\PP/\I$ are admissible, although it follows from proposition \ref{printro:factorization} that only the quotients $\PP/\I_n$ are modeled on vector spaces. Theorem \ref{thmintro:secondmain}\eqref{item3} shows that admissible tensor abelian categories are stable under many standard constructions such as subcategories, products, or functor categories.

\begin{thmintro}\label{thmintro:secondmain}
Let $\A$ be a tensor abelian category.
\begin{enumerate}[(1)]
\item If $\A$ is modeled on vector spaces, then $\A$ is  admissible.
\item\label{item2} If $\A$ is admissible, then every quotient $\A/\I$ by a tensor ideal $\I$ is admissible.
\item\label{item3} Assume that there is a family of tensor functors $(e_x:\A\to \B_x)_{x\in X}$ where the $\B_x$ are admissible tensor abelian categories, such that a diagram $0\to A'\to A\to A''\to 0$ in $\A$ is a short exact sequence if and only if 
\[0\to e_xA'\to e_xA\to e_xA''\to 0\]
is a short exact sequence in $\B_x$ for all $x\in X$. Then $\A$ is admissible.
\end{enumerate}
\end{thmintro}

Although theorem \ref{thmintro:secondmain} considerably improves theorem \ref{thmintro:main}, it does not provide a satisfactory description of the class of admissible tensor abelian categories. This leaves the following challenging questions open.
\begin{questionintro}
\begin{enumerate}[(A)]
\item Do there exist admissible categories which cannot be constructed from categories modeled on vector spaces by the operations \eqref{item2} and \eqref{item3} of theorem \ref{thmintro:secondmain}?
\item What are the examples of non-admissible tensor abelian categories in characteristic $2$? In odd characteristic, are there non-admissible tensor abelian categories which do not contain the tensor abelian category of finite dimensional super vector spaces  as a subquotient?
\end{enumerate}    
\end{questionintro}

\subsection{Relation with recent work of Coulembier \cite{Cou}} Let $\PP_d$ be the full abelian subcategory of $\PP$ whose objects are the subquotients of direct sums of $I^{\otimes d}$. In \cite{Cou}, Coulembier defines a category of \emph{universal polynomial functors} $\mathrm{UPol}_d$, which has $\PP_d$ as a quotient and which satisfies by definition a certain universal property with respect to all tensor categories in the sense of \cite{EGNO} (that is, with respect to the tensor abelian categories in the sense of the present article, which are in addition rigid and locally finite). The perspective of the present article rather goes in the opposite direction: instead of introducing a new  bigger category of universal polynomial functors, we provide a (restricted) universal property for the familiar category of strict polynomial functors. 

Our universal property is concretely connected to the universal polynomial functors of Coulembier as follows. For all objects $A$ of an admissible category $\A$, fix an exact tensor functor $\Phi^\A_A:\PP\to \A$ such that $\Phi^\A_A(I)=A$. By lemma \ref{lmintro:unique}, $\Phi_A^\A$ is natural with respect to $A$, hence for all strict polynomial functors $F$, there is a well-defined functor $F^\A:\A\to \A$ such that $F^\A(A)=\Phi_A^\A(F)$. Moreover for all tensor exact functors $T:\A\to \B$ the morphism $\id_{T(A)}$ extends to a unique isomorphism of tensor abelian functors $T\circ \Phi_A^\A\simeq \Phi_{T(A)}^\B$ which yields by evaluation on $F$ a natural isomorphism $\eta^T:T\circ F^\A\simeq F^\B\circ T$. One readily checks that the data of the $F^\A$ and the $\eta^T$ yield a universal polynomial functor in the sense of \cite[Section 8.2.1]{Cou}, albeit valid for all admissible tensor abelian categories $\A$ (instead of all rigid locally finite tensor abelian categories). 


\subsection{Strategy of proof and organization of the article} The remainder of the article is devoted to the proof of the results presented in this introduction. Let $\T$ be the full tensor subcategory of $\PP$ generated by $I$, and let $\G$ be the one generated by the divided powers $\Gamma^d$. There is an inclusion of tensor categories
\[\T\subset \G\subset \PP\;.\]
We construct the exact tensor functor $\Phi:\PP\to \A$ such that $\Phi(I)=A$ in several steps. We first construct a tensor functor $\Phi_\T:\T\to \A$ such that $\Phi(I)=A$. This first step essentially relies on  Schur-Weyl duality (see remark \ref{rk:schurweyl}). Then we extend (essentially by `left deriving') $\Phi_\T$ to a tensor functor $\Phi_\G:\G\to \A$ . Finally we prove (essentially by `right deriving') that such a tensor functor extends to a tensor functor $\Phi:\PP\to \A$ which is right exact. The condition that $\A$ is modeled on vector spaces is then the key technical point which ensures that $\Phi$ is actually exact.  
\begin{rkintro}
In view of the strategy of construction of $\Phi$, one might hope to obtain a universal property of $\PP$ with respect to all tensor abelian categories by considering only the right exact tensor functors $\Phi$. But that wouldn't do it, since there is no uniqueness of right exact tensor functors $\Phi:\PP\to \A$ such that $\Phi(I)=A$, as observed in remark \ref{rk:rightexactnonunique}.
\end{rkintro}

Section \ref{sec:tensabcat} recalls basic facts regarding tensor abelian categories and proves a general extension result in lemma \ref{lm:monextensions}, which allows to extend a tensor functor $\Phi:\G\to \A$ to a right exact tensor functor $\Phi:\PP\to \A$. Section \ref{sec:proplm} recalls the basic properties of the category $\PP$ which we will need, introduces the categories $\T$ and $\G$, and deduces from them the universal property in characteristic zero (given in proposition \ref{printro:univpropcharzero}) and the uniqueness lemma \ref{lmintro:unique}. Theorems \ref{thmintro:main} and \ref{thmintro:secondmain}, the main results, are proved at the same time in section \ref{sec:proofmain}. The last section deals with the proof of proposition \ref{printro:factorization}.

\subsection{Notations and terminology}
Throughout the article, we work over a fixed ground field $\kk$. Thus we say vector space, algebra, linear category, linear functor for $\kk$-vector space, $\kk$-algebra, $\kk$-linear categories, $\kk$-linear functors. We denote by $\Si_d$ the symmetric group on $d\ge 0$ letters, with $\Si_0=\Si_1=\{1\}$.

\section{Tensor abelian categories}\label{sec:tensabcat}
\subsection{Tensor categories}\label{subsec:recolltensab}
We take \cite[Chap 2 and 8]{EGNO} as a reference for symmetric monoidal categories, since the presentation given there minimizes the use of commutative diagrams. In the context of linear categories and linear functors over the ground field $\kk$, we will use the terms tensor categories and tensor functors instead of symmetric monoidal categories and symmetric monoidal functors (warning: our tensor categories are more general than the tensor categories of \cite[Chap 4]{EGNO} since we don't impose local finiteness, rigidity and triviality of $\End(\mathbf{1})$).

Thus, a \emph{tensor category} is a tuple $(\C,\otimes, \mathbf{1}, \alpha,\iota,\tau)$ where $\C$ is a linear category, $\otimes:\C\times\C\to \C$ is a bilinear functor, $\mathbf{1}$ is an object of $\C$ called the unit, and $\alpha: (x\otimes y)\otimes z\to x\otimes (y\otimes z)$, $\iota:\mathbf{1}\otimes \mathbf{1}\to \mathbf{1}$, and $\tau:x\otimes y\to y\otimes x$ are natural isomorphisms, which are required to satisfy the pentagon axiom and the unit axiom \cite[Def 2.1.1]{EGNO}, the hexagon axiom \cite[Def 8.1.1]{EGNO} and the identity $\tau^2=\id$. In addition, from the axioms of tensor categories, one can define \cite[Def 2.2.1]{EGNO} left and right unit isomorphisms $\lambda:\mathbf{1}\otimes x\to x$ and $\rho:x\otimes\mathbf{1}\to x$. A tensor category will often be simply denoted by a single letter $\C$. 

Given two tensor categories $\C$ and $\D$, a \emph{tensor functor} is a pair $(\Phi,\nabla)$ where $\Phi:\C\to \D$ is a linear functor such that $\Phi(\mathbf{1})$ is isomorphic to the unit of $\D$, and $\nabla:\Phi x\otimes\Phi y\to \Phi(x\otimes y)$ is a natural isomorphism such that the following two diagrams commute (for reasons of space and typography, we often write $\Phi x$ instead of $\Phi(x)$ if this causes no risk of confusion).
\[\begin{tikzcd}
(\Phi x\otimes\Phi y)\otimes \Phi z \ar{r}{\alpha}\ar{d}{\nabla\otimes \id} & \Phi x\otimes(\Phi y\otimes \Phi z)\ar{d}{\id\otimes\nabla}\\
\Phi(x\otimes y)\otimes \Phi z\ar{d}{\nabla} & \Phi x\otimes \Phi (y\otimes z)\ar{d}{\nabla}\\\Phi((x\otimes y)\otimes z)\ar{r}{\Phi(\alpha)}&\Phi(x\otimes (y\otimes z))
\end{tikzcd}
\qquad
\begin{tikzcd}
\Phi x\otimes \Phi y \ar{r}{\tau}\ar{d}{\nabla} & \Phi y\otimes \Phi x\ar{d}{\nabla}\\
\Phi(x \otimes y)\ar{r}{\Phi(\tau)} & \Phi(y\otimes x)
\end{tikzcd}
\]
We will usually denote a tensor functor by a single letter $\Phi$, and sometimes refer to $\nabla$ as the \emph{tensor structure on $\Phi$}

Given two tensor functors $\Phi,\Psi:\C\to \D$, a \emph{morphism of tensor functors} $f:\Phi\to \Psi$ is a natural transformation such that $f:\Phi\mathbf{1}\to \Psi\mathbf{1}$ is an isomorphism, and such that the following diagram commutes.
\[
\begin{tikzcd}
\Phi x\otimes\Phi y\ar{r}{f\otimes f}\ar{d}{\nabla}& \Psi x\otimes\Psi y \ar{d}{\nabla}\\
\Phi (x\otimes y)\ar{r}{f}& \Psi (x\otimes y)
\end{tikzcd}
\]

MacLane's strictness theorem \cite[Chap XI.3]{MLcat} shows that every tensor category can be replaced, up to an equivalence of tensor categories, by a strict one, i.e. by a tensor category in which the isomorphisms $\alpha$, $\iota$, $\lambda$, $\rho$ are identities. This strictness theorem considerably reduces the notational complexity of proofs. Therefore in the sequel of the article {\bf we will always consider without further notice that our tensor categories are strict}.
In particular, for all $d\ge 0$ every object $x$ has tensor powers $x^{\otimes d}$ (with $x^{\otimes 0}=\mathbf{1}$ as usual). There is a canonical morphism of monoids
\[\mathrm{can}:\Si_d\to \End_\C(x^{\otimes d})\]
such that for $d\ge 2$ and if $\sigma=(i,i+1)$ is the transposition which swaps $i$ and $i+1$, the map $\mathrm{can}(\sigma)$ equals $\id_{i-1}\otimes \tau\otimes \id_{d-i-1}$, where $\id_k$ is the identity of $x^{\otimes k}$. If $\Phi:\C\to \D$ is a tensor functor, there is a canonical isomorphism
\[\nabla_d:(\Phi x)^{\otimes d}\xrightarrow[]{\simeq} \Phi(x^{\otimes d})\]
such that $\nabla_1=\id$ and $\nabla_{d+1}=\nabla \circ (\nabla_d\otimes\id)$, and $\nabla_0:\mathbf{1}\to \Phi(\mathbf{1})$ is the unique isomorphism such that the following composition is the identity: 
\[\Phi x=\mathbf{1}\otimes \Phi x\xrightarrow[]{\nabla_0\otimes\id}\Phi \mathbf{1}\otimes \Phi x\xrightarrow[]{\nabla}\Phi(\mathbf{1}\otimes x)=\Phi x\;.\] 
It easily follows from the axioms that $\nabla_d$ is natural with respect to $x$ and $\Phi$, and that is preserves the action of $\Si_d$.

\subsection{An extension property for tensor abelian categories} We first recall a standard extension property for abelian categories, which is essentially given by the theory of right derived functors see e.g. \cite[Thm 8.7.2]{KS}:

\begin{lm}\label{lm:extensions}
Let $\A$ and $\B$ be two abelian linear categories, and let $\G\subset \A$ be a full subcategory of $\A$ whose objects are projective and such that every object of $\A$ is a quotient of a finite direct sum of objects of $\G$. 
Then the following holds. 
\begin{enumerate}
\item Every linear functor $\Phi:\G\to \B$ extends to a right exact linear functor $\overline{\Phi}:\A\to \B$.
\item Moreover, let us denote by $\Psi_{|\G}$ the restriction to $\G$ of an arbitrary linear functor $\Psi:\A\to \B$. Then every natural transformation $f:\Phi\to \Psi_{|\G}$ extends uniquely to a natural transformation $\overline{f}:\overline{\Phi}\to \Psi$.
\end{enumerate}
\end{lm}



A \emph{tensor abelian category} is tensor category which is abelian, and whose tensor product is exact with respect to each argument. In the context of tensor abelian categories, we have the following improvement of lemma \ref{lm:extensions}.

\begin{lm}\label{lm:monextensions}
In the context of lemma \ref{lm:extensions}, assume furthermore that $\A$ and $\B$ are tensor abelian categories, and that $\G$ is a tensor subcategory of $\A$.
\begin{enumerate}
\item If $\Phi$ is a tensor functor with tensor structure $\nabla$, then there is a unique tensor structure $\overline{\nabla}$ on $\overline{\Phi}$ extending $\nabla$.
\item If $\Psi:\A\to \B$ is a tensor functor and if $f:\Phi\to \Psi_{|\G}$ is a morphism of tensor functors, then $\overline{f}:\overline{\Phi}\to \Psi$ is also a morphism of tensor functors.
\end{enumerate}
\end{lm}
\begin{proof}
Notice that $\overline{\Phi} x\otimes \overline{\Phi} y$ is right exact with respect to each variable.

Therefore, $\nabla:\overline{\Phi} x\otimes\overline{\Phi} y\to \overline{\Phi}(x\otimes y)$, initially given only when $x$ and $y$ are objects of $\G$, can be extended by lemma \ref{lm:extensions}, first to all objects $x$ of $\A$ and $y$ of $\G$ and then to all objects $x$, $y$ of $\A$. The resulting two steps extension is denoted by $\overline{\nabla}$. Uniqueness of extensions of natural transformations whose domain is a right exact functor guaranties that $\overline{\nabla}$ inherits the properties of $\nabla$, i.e. is a tensor structure on $\overline{\Phi}$. (To be more specific, all diagrams involving $\nabla$ for $x$ and $y$ in $\G$ extend to diagrams involving $\overline{\nabla}$ for $x$ and $y$ in $\A$).

Similarly, if $\overline{f}:\overline{\Phi}\to \Psi$ extends a morphism of tensor functors, then by uniqueness of extensions of natural transformations from a right exact functor the natural transformations $\overline{\nabla}\circ (\overline{f}^{\otimes 2})$ and $\overline{f}\circ \overline{\nabla}$ are two extensions of the same natural transformation  $\nabla\circ ({f}^{\otimes 2})={f}\circ {\nabla}$, hence they are equal, which proves that $\overline{f}$ is a morphism of tensor functors.
\end{proof}

\section{Strict polynomial functors and the uniqueness lemma}\label{sec:proplm}

\subsection{Recollections of strict polynomial functors} 
Let $\F$ be the category whose objects are the endofunctors of finite dimensional vector spaces, and whose morphisms are the natural transformations. Then $\F$ is a linear abelian category whose direct sums, kernels and cokernels are computed objectwise. For example the kernel of $f:F\to G$ is the functor such that
$(\mathrm{Ker}\,f)(V)$ is the kernel of $f_V:F(V)\to G(V)$. The category $\F$ is endowed with a tensor product which is also defined objectwise $(F\otimes G)(V)=F(V)\otimes G(V)$ (the unit of the tensor product is the constant functor $\kk$), and which makes it a tensor abelian category.
The category $\PP$ of strict polynomial functors is an algebro-geometric analogue of $\F$. We refer the reader to \cite{FS} for definitions and proofs, we only recall here the properties that we will need. 

\subsubsection{}\label{subsubsec:1} The category $\PP$ is a tensor abelian category, and there is a faithful exact tensor functor $\U:\PP\to \F$. Moreover, a sequence $0\to F'\to F\to F''\to 0$ in $\PP$ is exact if and only if its image by $\U$ is exact in $\F$. As a consequence, for all finite dimensional vector spaces $V$ there is an exact tensor functor of evaluation on $V$: \[\begin{array}{cccc}
\ev_V:&\PP&\to &\Vect\\
& F &\mapsto & (\U F)(V)
\end{array}\]
and a sequence $0\to F'\to F\to F''\to 0$ in $\PP$ is exact if and only if its image by $\ev_V$ is an exact sequence of vector spaces for all $V$. In particular, $\PP$ is a tensor abelian category modeled on vector spaces in the sense of definition \ref{defiintro:modeled}.
\subsubsection{}\label{subsubsec:2} There is a unique injective and projective strict polynomial functor $I$ such that $\ev_VI=V$ for all $V$. The tensor powers $I^{\otimes d}$ are both injective and projective. We have $\Hom_\PP(I^{\otimes d},I^{\otimes e})=0$ if $d\ne e$, and for all $d$ the canonical morphism $\Si_d\to \End_\PP(I^{\otimes d})$ induces an isomorphism of algebras, see e.g. \cite[Lemma 2.1]{Touens}:
\[\kk\Si_d\xrightarrow[]{\simeq}\End_\PP(I^{\otimes d})\;.\]
\begin{rk}\label{rk:schurweyl}
Let $S(n,d)$ denote the classical Schur algebra as in \cite{Mar}. It follows from \cite[Thm 3.2]{FS} that the algebra $\End_\PP(I^{\otimes d})$ is isomorphic to $\End_{S(n,d)}((\kk^n)^{\otimes d})$ for $n\ge d$, hence the isomorphism above is nothing but the usual calculation underlying Schur-Weyl duality between Schur algebras and symmetric groups.   
\end{rk}
\subsubsection{}\label{subsubsec:3} The $d$-th divided power  $\Gamma^d$ is defined by $\Gamma^0=\kk$, $\Gamma^1=I$, and for all $d>1$, $\Gamma^d$ is the kernel of the morphism 
\[I^{\otimes d}\xrightarrow[]{\prod \id-\mathrm{can}(\sigma)}\bigoplus_{\sigma\in \Si_d} I^{\otimes d}\;.\]
For all tuples $\mu=(\mu_1,\dots,\mu_r)$ of nonnegative integers, we denote $\Gamma^\mu$ the tensor product $\Gamma^{\mu_1}\otimes\cdots\otimes\Gamma^{\mu_r}$. The $\Gamma^\mu$ are projective, and every object of $\PP$ is a quotient of a finite direct sum of such projectives.
 If the weight $|\mu|:=\sum \mu_i$ of $\mu$ equals $d$, there is a corresponding Young subgroup $\Si_\mu\subset \Si_d$, and an exact sequence:
\begin{equation}\label{eqn:copres}
0\to \Gamma^\mu\xrightarrow[]{i_\mu}I^{\otimes d}\xrightarrow[]{\prod \id-\mathrm{can}(\sigma)}\bigoplus_{\sigma\in \Si_\mu} I^{\otimes d}\;.
\end{equation}

\subsubsection{}\label{subsubsec:4} In characteristic zero, every object of $\PP$ is a finite direct sum of simple objects, and every simple object is a direct summand of a tensor power $I^{\otimes d}$.

\subsection{The subcategory $\T$ and the proof of proposition \ref{printro:univpropcharzero}}
Let $\T$ be the full tensor subcategory of $\PP$ generated by $I$. Thus $\T$ is the full subcategory of $\PP$ whose objects are the tensor powers of $I$.
\begin{lm}\label{lm:extensionT}
Let $\A$ be a tensor category. For all objects $A$ of $\A$ there is a tensor functor $\Phi:\T\to \A$ such that $\Phi(I)=A$. Moreover, if $\Phi,\Psi:\T\to \A$ are two tensor functors, then every morphism $f:\Phi(I)\to \Psi(I)$ extends to a unique morphism of tensor functors $\overline{f}:\Phi\to \Psi$.
\end{lm}
\begin{proof}
We define a functor $\Phi:\T\to \A$ on objects by $\Phi(I^{\otimes d})=A^{\otimes d}$. In view of the description of morphisms in $\T$ given \ref{subsubsec:2}, in order to define $\Phi$ on morphisms it suffices to specify $\Phi(\mathrm{can}(\sigma))$ for all $\sigma\in\Si_d$. We let $\Phi(\mathrm{can}(\sigma))=\mathrm{can}(\sigma)$ where the canonical map on the right hand side is the one of $\A$. Since $\mathrm{can}$ is a morphism of monoids, $\Phi$ is indeed a linear functor. It becomes a tensor functor with the tensor monoidal structure $\nabla:=\id$. 

Next we prove the extension property of morphisms. Let $\overline{f}:\Phi\to \Psi$ be a morphism of tensor functors whose value on $I$ equals $f:\Phi(I)\to \Psi(I)$. Naturality of $\nabla_d$ with respect to $\Phi$ implies that $\overline{f}:\Phi(I^{\otimes d})\to \Psi(I^{\otimes d})$ equals the composition
\begin{equation}\label{eqn:Text1}
\Phi(I^{\otimes d})\xrightarrow[\simeq]{\nabla_d^{-1}}\Phi(I)^{\otimes d}\xrightarrow[]{f^{\otimes d}}\Psi(I)^{\otimes d}
\xrightarrow[\simeq]{\nabla_d}\Psi(I^{\otimes d})\;.
\end{equation}
Thus $\overline{f}$ is uniquely determined by $f^{\otimes d}$, which proves uniqueness of $\overline{f}$. Conversely, assume that $f:\Phi(I)\to \Psi(I)$ is given. For all $d\ge 0$ we take the composition \eqref{eqn:Text1} as the definition of $\overline{f}:\Phi(I^{\otimes d})\to \Psi(I^{\otimes d})$. Since all the maps of this composition are $\Si_d$-equivariant, our definition of $\overline{f}$ actually yields a natural transformation. 
It remains to show that $\overline{f}$ is a morphism of tensor functors, which can be deduced from \eqref{eqn:Text1} applied to a tensor product $(I^{\otimes d}\otimes I^{\otimes e})$ by noting that $\nabla_{d+e}=\nabla\circ (\nabla_d\otimes\nabla_e)$.
\end{proof}

Now we can prove the universal property of $\PP$ in characteristic zero. 
\begin{proof}[Proof of proposition \ref{printro:univpropcharzero}]
Let $\A$ be a tensor abelian category and let $A$ be an object of $\A$. By lemma \ref{lm:extensionT}, there is a unique tensor functor $\Phi_\T:\T\to \A$ such that $\Phi_\T(I)=A$. In characteristic zero, every object of $\PP$ is a quotient of a finite direct sum of objects of $\T$ by \ref{subsubsec:4}. Hence lemma \ref{lm:monextensions} says that there is a unique right exact tensor functor $\Phi:\PP\to \A$ extending $\Phi_\T$. Finally, since $\PP$ is semi-simple in characteristic zero, every linear functor $\PP\to\A$ is exact, hence $\Phi$ is exact, which finishes the proof.
\end{proof}

\subsection{The subcategory $\G$ and the proof of lemma \ref{lmintro:unique}}\label{subsec:prooflemmauniqueness}
Let $\G$ be the full tensor subcategory of $\PP$ generated by the divided powers $\Gamma^d$ for $d\ge 0$. Thus the objects of $\G$ are the $\Gamma^\mu$ for all tuples $\mu$. There are inclusions of tensor categories:
\[\T\subset \G\subset \PP\;.\]
\begin{proof}[Proof of lemma \ref{lmintro:unique}]
Let $\A$ be a tensor abelian category, let $\Phi,\Psi:\PP\to \A$ be exact tensor functors, and let $f:\Phi(I)\to \Psi(I)$ be a morphism in $\A$. We denote by $\Phi_{|\C}$ and $\Psi_{|\C}$ the restrictions of $\Phi$ and $\Psi$ to a subcategory $\C\subset \PP$. 
By lemma \ref{lm:extensionT}, $f$ extends uniquely to a morphism of tensor functors $f_\T:\Phi_{|\T}\to \Psi_{|\T}$. 

Now we claim that $f_\T$ extends uniquely to a morphism of tensor functors $f_\G:\Phi_{|\G}\to \Psi_{|\G}$. For this purpose we use in an essential way that $\Phi$ and $\Psi$ are left exact. Indeed, assume first that such an extension $f_\G$ is given. For all tuples $\mu$ of weight $d$ we a commutative square whose horizontal rows are exact:
\begin{equation}\label{eqn:extU}
\begin{tikzcd}[column sep=large]
0\to\Phi(\Gamma^\mu)\ar{d}{f_\G}\ar{r}{\Phi(i_\mu)}&  \Phi(I^{\otimes d})\ar{d}{f_\T}
\ar{rr}{\prod \id-\Phi(\mathrm{can}(\sigma))}&&\bigoplus_{\sigma\in\Si_\mu}\Phi(I^{\otimes d})\ar{d}[swap]{\bigoplus f_\T}\\
0\to\Psi(\Gamma^\mu)\ar{r}{\Psi(i_\mu)}&  \Psi(I^{\otimes d})\ar{rr}{\prod \id-\Psi(\mathrm{can}(\sigma))}&&\bigoplus_{\sigma\in\Si_\mu}\Psi(I^{\otimes d})
\end{tikzcd}.
\end{equation}
This shows that $f_\G$ is uniquely determined by $f_\T$. Next, let us prove that every morphism of tensor functors $f_\T$ has an extension to $\G$. For all $\mu$ we define a morphism $f_\G:\Phi(\Gamma^\mu)\to \Psi(\Gamma^\mu)$ as the unique morphism such that diagram \eqref{eqn:extU} commutes. In order to check that these morphisms actually yield a morphism of tensor functors, we consider two objects $\Gamma^\mu$ and $\Gamma^\nu$ associated to tuples of weight $|\mu|=d$ and $|\nu|=e$, and a morphism $u:\Gamma^\mu\to \Gamma^\nu$. Since the tensor powers of $I$ are injective, $u$ extends to a morphism $\overline{u}:I^{\otimes d}\to I^{\otimes e}$ such that $i_\nu\circ u=\overline{u}\circ i_\mu$. Since $f_\T$ is a morphism of tensor functors we have commutative squares:
\begin{equation}\label{eqn:squaresU}
\begin{tikzcd}
\Phi(I^{\otimes d})\ar{r}{\Phi(\overline{u})}\ar{d}[swap]{f_\T} &\Phi(I^{\otimes e}) \ar{d}{f_\T}\\
\Psi(I^{\otimes d})\ar{r}{\Psi(\overline{u})}&\Psi(I^{\otimes e})
\end{tikzcd}
\qquad
\begin{tikzcd}[column sep=large]
\Phi(I^{\otimes d})\otimes\Phi(I^{\otimes e})\ar{r}{f_\T\otimes f_\T}\ar{d}[swap]{\nabla} &\Psi(I^{\otimes d})\otimes\Psi(I^{\otimes e}) \ar{d}{\nabla}\\
\Phi(I^{\otimes d}\otimes I^{\otimes e})\ar{r}{f_\T}&\Psi(I^{\otimes d}\otimes I^{\otimes e})
\end{tikzcd}.    
\end{equation}
The squares \eqref{eqn:squaresU} contain the squares \eqref{eqn:squaresUbis} below as subdiagrams. To be more specific, thanks to $i_\mu$ and $i_\nu$, the vertices of the squares \eqref{eqn:squaresUbis} are subobjects of the vertices of the squares \eqref{eqn:squaresU}, and the arrows the squares \eqref{eqn:squaresUbis} are the restrictions of the arrows of the squares  \eqref{eqn:squaresU}.
\begin{equation}\label{eqn:squaresUbis}
\begin{tikzcd}
\Phi(\Gamma^\mu)\ar{r}{\Phi(u)}\ar{d}[swap]{f_\G} &\Phi(\Gamma^\nu) \ar{d}{f_\G}\\
\Psi(\Gamma^\mu)\ar{r}{\Psi(u)}&\Psi(\Gamma^\nu)
\end{tikzcd}
\qquad
\begin{tikzcd}[column sep=large]
\Phi(\Gamma^\mu)\otimes\Phi(\Gamma^\nu)\ar{r}{f_\G\otimes f_\G}\ar{d}[swap]{\nabla} &\Psi(\Gamma^\mu)\otimes\Psi(\Gamma^\nu) \ar{d}{\nabla}\\
\Phi(\Gamma^\mu\otimes \Gamma^\nu)\ar{r}{f_\G}&\Psi(\Gamma^\mu\otimes \Gamma^\nu)
\end{tikzcd}.    
\end{equation}
The commutativity of the squares \eqref{eqn:squaresU} imply the commutativity of their subsquares \eqref{eqn:squaresUbis}. This shows that $f_\G$ is indeed a morphism of tensor functors.

Finally, since every object of $\PP$ is the quotient of a finite direct sum of objects of $\G$ by \ref{subsubsec:3},  lemma \ref{lm:monextensions} implies that $f_\G$ extends uniquely to a morphism of tensor functors $\overline{f}:\Phi\to \Psi$, which finishes the proof of lemma \ref{lmintro:unique}.
\end{proof}

\begin{rk}\label{rk:rightexactnonunique}
Lemma \ref{lmintro:unique} implies that two exact tensor functors $\Phi,\Psi:\PP\to \A$ are isomorphic if and only if $\Phi(I)$ and $\Psi(I)$ are isomorphic. This fails for half exact tensor functors in positive characteristic, even for $\A=\Vect$. For example, the functor $\Psi(F)=\bigoplus_{d\ge 0}\Hom_\PP(S^{d},F)$ can be equipped \cite[Lemma 5.5]{Touaif} with a tensor structure. We have $\Psi(I)\simeq\kk =\ev_\kk(I)$, although $\Psi$ is not isomorphic to $\ev_\kk$: the former is only left exact whereas the latter is exact. Similarly, $\Phi(F)=\bigoplus_{d\ge 0}\Hom_\PP(F,\Gamma^d)^*$ is a right exact tensor functor with $\Phi(I)=\kk$ though not isomorphic to $\ev_\kk$.
\end{rk}

\section{Admissible categories: proof of theorems 5 and 9}\label{sec:proofmain}

Theorems 5 and 9 are direct consequences of the following theorem. Indeed, statements \eqref{item1bis} and \eqref{item3bis} of theorem \ref{thm:maincorps} imply that tensor abelian categories modeled on vector spaces are admissible.  
\begin{thm}\label{thm:maincorps}
Let $\A$ be a tensor abelian category.
\begin{enumerate}[(1)]
\item\label{item1bis} The tensor abelian category $\Vect$ of all vector spaces (with the usual tensor product and the usual braiding) is admissible.
\item\label{item2bis} If $\A$ is admissible, then every quotient $\A/\I$ by a tensor ideal $\I$ is admissible.
\item\label{item3bis} Assume that there is a family of tensor functors $(e_x:\A\to \B_x)_{x\in X}$ where the $\B_x$ are admissible tensor abelian categories, such that a diagram $0\to A'\to A\to A''\to 0$ in $\A$ is a short exact sequence if and only if 
\[0\to e_xA'\to e_xA\to e_xA''\to 0\]
is a short exact sequence in $\B_x$ for all $x\in X$. Then $\A$ is admissible.
\end{enumerate}
\end{thm}
The remainder of the section is devoted to the proof of theorem \ref{thm:maincorps}, which will be done through a series of lemmas. 
The following lemma proves that $\Vect$ is admissible, that is, assertion \eqref{item1bis} of theorem \ref{thm:maincorps}. 
\begin{lm}\label{lm:classifVect}
For all vector spaces $V$ there exists an exact tensor functor $\ev_V:\PP\to \Vect$ such that $\ev_V(I)=V$.    
\end{lm}
\begin{proof}
If $V$ is finite dimensional, then $\ev_V$ is nothing but the exact tensor functor described in section \ref{subsubsec:1}. If $V$ is infinite dimensional, we define the exact functor $\ev_V$ by the formula $\ev_V(F):=\colim_U \ev_U(F)$ where the colimit is taken over the filtered poset of finite dimensional vector subspaces $U\subset V$ ordered by inclusion. The tensor structure of $\ev_V$ is obtained similarly by taking the colimit over $U$ of the tensor structures of $\ev_U$ and by noticing that $\colim_U (\ev_U(F)\otimes ev_U(G))$ is canonically isomorphic to  $\ev_V(F)\otimes \ev_V(G)$.   
\end{proof}

The following lemma proves assertion \eqref{item2bis} of theorem \ref{thm:maincorps}.

\begin{lm}
If $\I$ is a tensor ideal of an admissible tensor abelian category $\A$, then for all objects $A$ of $\A/\I$ there is an exact tensor functor $\Phi:\PP\to \A/\I$ such that $\Phi(I)=A$.    
\end{lm}
\begin{proof}
By definition $\A/\I$ and $\A$ have the same objects. Since $\A$ is admissible, there is an exact tensor functor $\overline{\Phi}:\PP\to \A$ such that $\overline{\Phi}(I)=A$. If $\pi:\A\to \A/\I$ denotes the quotient map, then $\Phi=\pi\circ\overline{\Phi}$ is the sought-after exact tensor functor.   
\end{proof}

It remains to prove the assertion \eqref{item3bis} of theorem \ref{thm:maincorps}. For all objects $A$ of an arbitrary tensor abelian category $\A$, lemma \ref{lm:extensionT} gives a tensor functor $\Phi_\T:\T\to \A$ such that $\Phi(I)=A$. By the following lemma, this tensor functor extends to $\G$.
\begin{lm}\label{lm:extTG}
Let $\A$ be an arbitrary tensor abelian category. Every tensor functor $\Phi_\T:\T\to \A$ extends to a tensor functor $\Phi_\G:\G\to \A$.
\end{lm}
\begin{proof}
Let us rewrite as $0\to \Gamma^\mu \xrightarrow[]{i_\mu}J^0_\mu \xrightarrow[]{d_\mu} J^1_\mu$ the injective copresentation of $\Gamma^\mu$ by direct sums of copies of $I^{\otimes d}$ given by $\eqref{eqn:copres}$. Every morphism $u:\Gamma^\mu\to \Gamma^\nu$ lifts to a morphism $u^\bullet=(u^0,u^1)$ of copresentations, unique up to homotopy.  
Let $\T^\oplus$ be the full subcategory of $\PP$ whose objects are direct sums of copies of tensor powers. The linear functor $\Phi_\T$ extends to $\T^\oplus$, and we still denote by $\Phi_\T$ the resulting functor.

We can now define the linear functor $\Phi_\G:\G\to \A$ by `right deriving' $\Phi_\T$. Namely, we define $\Phi_\G(\Gamma^\mu):=\mathrm{Ker}\,\Phi_\T(d_\mu)$. Note that the kernel is only unique up to isomorphism, so we choose one representative for each $\mu$. For $\Gamma^\mu=I^{\otimes d}$ we choose it in order that $\Phi_\G(I^{\otimes d})=\Phi_\T(I^{\otimes d})$. For all morphisms $u:\Gamma^\mu\to \Gamma^\nu$, we define $\Phi_\G(f):= H^0(\Phi_\T(u^\bullet))$. Since $u^\bullet$ is unique up to homotopy  and since $\Phi_\T$ is $\kk$-linear, $\Phi_\G(u)$ does not depend on the choice of $u^\bullet$. This implies that $\Phi_\G(u)$ is well-defined and that $\Phi_\G$ is indeed a linear functor extending $\Phi_\T$. Note that by definition, $\Phi_\G(i_\mu)$ is the canonical inclusion of $\Phi_\G(\Gamma^\mu)=\mathrm{Ker}\,\Phi_\T(d_\mu)$ into $\Phi_\T(J^0_\mu)$.

It remains to define the tensor structure of $\Phi_\G$. For all tuples $\mu$ and $\nu$, we define $\nabla_\G$ as the unique map fitting into the following diagram, whose exact rows are induced by the maps $i_\mu$, $i_\nu$, $d_\mu$ and $d_\nu$:
\begin{equation*}
\begin{tikzcd}
0\to \Phi_\G(\Gamma^\mu)\otimes \Phi_\G(\Gamma^\nu)\ar{r}\ar{d}{\nabla_\G}& \Phi_\T(J^0_\mu)\otimes \Phi_\T(J^0_\nu)\ar{r}\ar{d}{\nabla_\T}&  
{\begin{array}{c}\Phi_\T(J^0_\mu)\otimes \Phi_\T(J^1_\nu)\\
\oplus \Phi_\T(J^1_\mu)\otimes \Phi_\T(J^0_\nu)
\end{array}}\ar{d}{\nabla_\T\oplus \nabla_\T}\\
0\to \Phi_\G(\Gamma^\mu\otimes\Gamma^\nu)\ar{r}& \Phi_\T(J^0_\mu\otimes J^0_\nu)\ar{r}&  
{\begin{array}{c}\Phi_\T(J^0_\mu\otimes J^1_\nu)\\
\oplus \Phi_\T(J^1_\mu\otimes J^0_\nu)
\end{array}}
\end{tikzcd}
\end{equation*}
Then $\nabla_\G$ extends $\nabla_\T$, that is, if $\Gamma^\mu= I^{\otimes d}$ and $\Gamma^\nu=I^{\otimes e}$, then $\nabla_\G=\nabla_\T$.
One checks that $\nabla_\G$ is natural with respect to $\Gamma^\mu$ and $\Gamma^\nu$ and satisfies the axioms of a tensor structure (that is, the identities $\nabla\circ(\nabla\otimes\id)=\nabla\circ(\id\otimes\nabla)$ and $\nabla\circ\tau=\Phi(\tau)\circ \nabla$) in the same way as one checks the properties of $f_\G$ in the proof of lemma \ref{lmintro:unique}. To be more specific, the left hand side square of the diagram defining $\nabla_\G$ expresses it as the restriction of $\nabla_\T$ to the subobjects $\Phi_\G(\Gamma^\mu)\otimes \Phi_\G(\Gamma^\nu)$ and $\Phi_\G(\Gamma^\mu\otimes\Gamma^\nu)$, and the properties of $\nabla_\G$ follow by restricting the diagrams expressing that $\nabla_\T$ is natural in $J^0_\mu$ and $J^0_\nu$ and satisfies the axioms of a tensor structure.   
\end{proof}

Lemma \ref{lm:monextensions} shows that every tensor functor $\Phi_\G:\G\to \A$ extends to a right exact tensor functor $\Phi:\PP\to \A$. In general, even for the nicest case $\A=\Vect$, there is no reason why $\Phi$ should be exact (as a concrete example of this phenomenon, take e.g. the right exact tensor functor $\Phi$ of remark \ref{rk:rightexactnonunique} and $\Phi_\G$ its restriction to $\G$). Nonetheless, if $\A$ satisfies the hypothesis given in theorem \ref{thm:maincorps}\eqref{item3bis}, then it is possible to find a certain extension $\Phi_\G:\G\to \A$ (actually the specific one contructed in the proof of lemma \ref{lm:extTG}) which in turn extends to an exact functor $\Phi:\PP\to \A$. 
This is the content of the following lemma, which finishes the proof of theorem \ref{thm:maincorps}.

\begin{lm}
Assume that there is a family of tensor functors $(e_x:\A\to \B_x)_{x\in X}$ where the $\B_x$ are admissible tensor abelian categories, such that a diagram $0\to A'\to A\to A''\to 0$ in $\A$ is a short exact sequence if and only if $e_x$ sends it to a short exact sequence in $\B_x$ for all $x\in X$. Then every tensor functor $\Phi_\T:\T\to \A$ extends to an exact tensor functor $\Phi:\PP\to \A$.    
\end{lm}
\begin{proof}
Let $\Phi_\G$ be the extension of $\Phi_\T$ constructed in the proof of lemma \ref{lm:extTG}. 
Since the $\B_x$ are admissible, for all $x$ there is a (unique) exact tensor functor $\Phi_x:\PP\to \B_x$ such that $\Phi_x(I)=e_x(\Phi_\T(I))$. Let $\Phi_{x|\C}:\C\to \B_x$ denote the restriction of $\Phi_x$ to a subcategory $\C$ of $\PP$.  We first claim that there is an isomorphism of tensor functors:
\begin{equation}\label{eqn:isoclef}
f_\G:\Phi_{x|\G} \xrightarrow[]{\simeq} e_x\circ \Phi_\G\;.
\end{equation}
Indeed, lemma \ref{lm:extensionT} yields an isomorphism of tensor functors $f_\T:\Phi_{x|\T}\xrightarrow[]{\simeq} e_x\circ \Phi_\T$. This latter extends to an isomorphism of tensor functors $f_\G$ by the same reasoning as in the proof of lemma \ref{lmintro:unique}. Namely, we define the isomorphism $f_\G$ as the unique map fitting into the commutative diagram with exact rows:
\begin{equation*}
\begin{tikzcd}
0\to \Phi_{x}(\Gamma^\mu)\ar{r}\ar[dashed]{d}{f_\G}[swap]{\simeq}& \Phi_x(J^0_\mu)\ar{r}\ar{d}{f_\T}[swap]{\simeq}&\Phi_x(J^1_\mu)\ar{d}{f_\T}\\
0\to e_x(\Phi_\G(\Gamma^\mu))\ar{r}& e_x(\Phi_\T(J^0_\mu))\ar{r}&e_x(\Phi_\T(J^1_\mu))
\end{tikzcd}.
\end{equation*}
(In this diagram $0\to \Gamma^\mu\to J^0_\mu\to J_1^\mu$ denotes the injective copresentation of $\Gamma^\mu$ given by \eqref{eqn:copres}, and $\Phi_\T$ denotes the extension of $\Phi_\T$ to $\T^\oplus$. Exactness of the rows follow from exactness of $\Phi_x$ and $e_x$ and the construction of $\Phi_\G$ in the proof of lemma \ref{lm:extTG}.) The left hand side square of the diagram exhibits $f_\G$ as the restriction of the isomorphism $f_\T:\Phi_x(J^0_\mu)\to e_x(\Phi_\T)$ to the subobjects $\Phi_x(\Gamma^\mu)$ and $e_x(\Phi_\G(\Gamma^\mu))$. The fact that $f_\G$ is natural with respect to $\Gamma^\mu$ and that it commutes with $\nabla$ follows by restricting to $f_\G$ the diagrams expressing that $f_\T$ is   natural with respect to $J^0_\mu$ and that $f_\T$ commutes with $\nabla$.

Let us now denote by $\Phi:\PP\to \A$ the right exact extension of $\Phi_\G$ provided by lemma \ref{lm:monextensions}. The right exact tensor functors $\Phi_x$ and $e_x\circ \Phi$ have isomorphic restrictions to $\G$ by the isomorphism \eqref{eqn:isoclef} hence they are isomorphic by lemma \ref{lm:monextensions}. Thus $e_x\circ \Phi$ is exact for all $x$. Since the functors $e_x$ reflect exactness, this implies that the tensor functor $\Phi$ is exact.  
\end{proof}

\section{Proof of proposition \ref{printro:factorization}}
Let $\A$ is a category modeled on vector spaces, and let us fix a family of test functors $(e_x:\A\to \Vect)_{x\in X}$. Thus the $e_x$ are tensor functors, and the functor 
\[e:=\prod_{x\in X}e_x\,:\, \A\to \prod_{x\in X}\Vect\]
reflects exactness. 
Let $\Phi:\PP\to \A$ be an arbitrary exact tensor functor. By lemmas \ref{lm:classifVect} and \ref{lmintro:unique}, each exact tensor functor $e_x\circ \Phi$ is isomorphic to $\ev_{V_x}$ where $V_x= e_x(\Phi(I))$. 
Since $e$ reflects exactness, it is faithful, hence we have: 
\begin{equation}\label{eqn:Ker}
\mathrm{Ker}\,\Phi = \mathrm{Ker}\, e\circ \Phi = \bigcap_{x\in X}\mathrm{Ker}\,\ev_{V_x}\;.    
\end{equation}
Recall that $\I_d=\mathrm{Ker}\,\ev_{\kk^d}$ for all integers $d$, and that $\I_{+\infty}:=0=\bigcap_{d\ge 0} \I_d$. Notice that if $U$ is a direct summand of $V$ then $\ev_U$ is a retract of $\ev_V$, hence $\mathrm{Ker}\,\ev_V \subset \mathrm{Ker}\,\ev_U$. 
This implies that the right hand side of \eqref{eqn:Ker} equals $\I_n$, where $n=\sup\{\dim V_x\,|\, x\in X\}$, which finishes the proof of proposition \ref{printro:factorization}.

\printbibliography

\end{document}